\documentclass{amsart}
\usepackage{amsmath, amsthm, amsfonts}
\newtheorem{defi}{Definition}[section]

\newtheorem{thm}{Theorem}[section]
\newtheorem{prop}{Proposition}[section]
\newtheorem{cor}{Corollary}[section]
\newtheorem{rem}{Remark}[section]

\usepackage[cp1250]{inputenc}
\usepackage{amssymb}
\usepackage[T1]{fontenc}

\usepackage{graphicx}
\usepackage{amssymb}
\usepackage{times}
\usepackage{csquotes}
\usepackage{yfonts}
\usepackage{bbm}
\usepackage{dsfont}
\usepackage{enumitem}
\usepackage{bbm}
\usepackage{xcolor}
\usepackage{float}

\title[Extreme coherent distributions with no atoms]{On the existence of extreme coherent \\ distributions with no atoms}

\author{Stanis\l{}aw Cichomski}
\address{Department of Mathematics, Informatics and Mechanics\\
 University of Warsaw\\
Banacha 2, 02-097 Warsaw\\
Poland}

\author{Adam Os\k{e}kowski}
\address{Department of Mathematics, Informatics and Mechanics\\
 University of Warsaw\\
Banacha 2, 02-097 Warsaw\\
Poland}

\numberwithin{equation}{section}

\begin{document}

\begin{abstract} The paper is devoted to the study of extremal points of $\mathcal{C}$, the family of all two-variate coherent distributions on $[0,1]^2$. It is well-known that the set $\mathcal{C}$ is convex and weak$^*$ compact, and all extreme points of $\mathcal{C}$
must be supported on sets of Lebesgue measure zero. Conversely, examples of extreme coherent measures, with a finite or countable infinite number of atoms, have been successfully constructed in the literature. The main purpose of this article is to bridge the natural gap between those two results: we provide an example of extreme coherent distribution with an uncountable support and with no atoms. Our argument is based on classical tools and ideas from the dynamical systems theory. This unexpected connection can be regarded as an independent contribution of the paper.
\end{abstract}

\maketitle

\section{Introduction}

The concept of coherent random variables and distributions appeared for the first time in the mid-nineties, in the paper by Dawid et al. \cite{C1}.  Within a Bayesian inference model, they investigated practical procedures for combining distinct conditional probabilities of a fixed uncertain event into a single and reliable forecast. Honoring their original definition, we say that a multivariate random vector $(X_1, X_2, \dots, X_n )$ defined on some arbitrary probability space $(\Omega, \mathcal{F}, \mathbb{P})$ is coherent, if there is an event $E\in \mathcal{F}$ and sequence of sub-$\sigma$-fields $\mathcal{G}_1, \mathcal{G}_2, \dots, \mathcal{G}_n \subset \mathcal{F}$ such that
\begin{equation} \label{experts} X_i=\mathbb{P}(E|\mathcal{G}_i) \ \ \ \ \text{for all} \ \ \ \ i=1,2,\dots n, \ \ \ \  \text{almost surely}.\end{equation}
The motivating interpretation is as follows: suppose that a group of $n$ experts provides their personal estimates on the likelihood of an uncertain event $E$ of our interest. Moreover, assume that the entire knowledge of  $j$-th specialist ($j=1,2,\dots,n$) is strictly restricted to  some 
personalized source of information (sub-$\sigma$-algebra) $\mathcal{G}_j$  and his predictions are solely based on this limited knowledge.
Then, in such an idealized model, the obtained sequence $(X_1, X_2, \dots, X_n)$ of experts' opinions on the odds of event $E$, clearly satisfies (\ref{experts}).
In such a case, the joint distribution of vector $(X_i)_{i=1}^n$ on the unit cube $[0,1]^n$ is also called coherent. 
In what follows, we write $m \in \mathcal{C}_n$ to indicate that a probability measure $m$ is the distribution of some $n$-variate coherent vector  $(X_i)_{i=1}^n$. 

After two decades of hiatus, the mathematical study of coherence has been reinstated by the influential paper of  Burdzy and Pitman, see \cite{pitman}.
For the latest advancements in the study of multivariate coherent distributions, see e.g. \cite{EJP, contra2,  max}.  Modernly, the notion of coherent distributions is well suited to the analysis of information diverisity in social networks and stochastic  modelling of other data-driven technologies. Quite unexpectedly, it fits efficiently into the rapidly growing framework of statistical learning, game theory and  microeconomics. More specifically, we refer the interested reader to the works on optimal combining \cite{C2, C4, C3}, posterior feasibility \cite{B4, B1, kDoob}, Bayesian rationality \cite{BR1, BR2, BR3}, maximal discrepancy \cite{contra, EJP,  BPC} and Bayesian persuasion \cite{Arieli, B2, Kamenica}.

\smallskip
Just as importantly as its applications in statistics and microeconomics, the notion of coherent distributions is a valuable mathematical concept in its own right.
Its distinctive feature is that it appears at the intersection of two profound classical fields: probability theory and discrete mathematics.
On the purely probabilistic side, coherent distributions are closely related to permutons and copulas \cite{asymp, EDSM, zhu}, study of probabilistic measures with given marginals \cite{B2, Gutmann, Strassen}, or theory of martingales indexed by partially ordered sets \cite{pitman, kDoob, max}.
From the combinatorial perspective, coherent distributions are strongly connected with graph theory \cite{mastersthesis, BPC, tao},
mathematical tomography \cite{GaleRyser, Gutmann, PPinfo} and combinatorial matrix theory \cite{brualdi, mastersthesis, ryser}.
Our work here reveals an unexpected connection with another field of mathematics: below, we will exhibit the link between the geometric analysis of coherent distributions with the theory of dynamical systems.

Although most of the existing literature deals with some various optimization problems, ongoing research is strictly constricted by the fact that  geometric structure of coherent distributions is very complex and not entirely explored. Let us briefly recall what is known. From now on, we will drop the subscript '$2$' and write  $\mathcal{C}$ instead of $\mathcal{C}_2$ for the foundational class of all  two-dimensional coherent distributions on $[0,1]^2$.
As demonstrated in \cite{pitman}, the family $\mathcal{C}$ is  a convex and weak$^*$ compact subset of all probability distributions on $[0,1]^2$.
Studying $\mathrm{ext}(\mathcal{C})$, the set  of extreme points of $\mathcal{C}$, is consequently the primary challenge in this field.
As shown independently in \cite{B1, zhu}, there exist extreme coherent measures with arbitrary large
or even countable infinite number of atoms. For a useful comparison,  note that martingales $(M_1, M_2)$, whose distributions are extremal, are always supported on one or two points, see \cite{extmart}. Next, as verified in \cite{B1}, elements of $\mathrm{ext}(\mathcal{C})$ are always supported on sets of Lebesgue measure zero. Lastly, in \cite{asymp}, a complete characterisation of $\mathrm{ext}(\mathcal{C})$ was developed, but this did not result in any easily verifiable criterion.
 
\begin{figure}[H] 
\centering
\includegraphics[width=45mm]{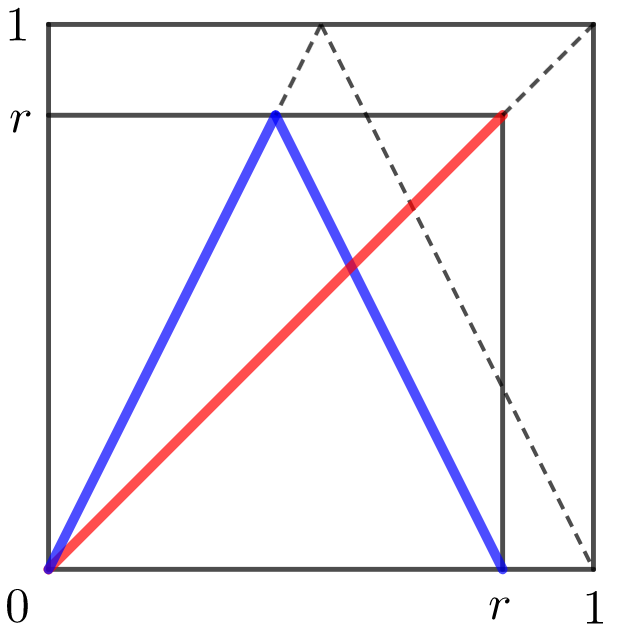}
\vspace{0.2cm}
\caption{Red color stands for $\Delta_r$, the support of measure $\mu_r$. Blue color represents  $\Gamma_r$, the support of measure $\nu_r$.
The union $\Delta_r\cup \Gamma_r$ is the support of an extreme coherent distribution $m_r$.}
\label{rys1} 
\end{figure}

The main aim of this article is to tackle the natural and intriguing problem formulated in \cite{B1, asymp, zhu}: does there exist any non-atomic extreme points of $\mathcal{C}$? We will answer this question in the affirmative. Beforehand, we need to develop some auxiliary notation. For $r\in(0,1)$, introduce $t_r:[0,r]\rightarrow [0,r]$ by 
$$t_r(x)=2\min\{x, r-x\}$$
and distinguish the two special sets
$$\Delta_r=\left\{(x,x): x\in [0,r]\right\} \ \ \ \ \text{and} \ \ \ \ \Gamma_r= \left\{(x,t_r(x)): x\in [0,r]\right\}.$$ 
Next, consider nonnegative Borel measures $\mu_r$ and $\nu_r$, supported on $\Delta_r$ and $\Gamma_r$, respectively, which are defined by
\begin{equation} \label{mur-defi} \mu_r\left(\left\{(s,s): 0\le s \le x \right\} \right) \ \ = \ \ c_r^{-1}\int_{0}^x\frac{s}{1-s} \ \mathrm{d}s,\end{equation}
\begin{equation} \label{nur-defi} \nu_r\left(\left\{\left(s,t_r(s)\right): 0\le s \le x \right\} \right) \ \ = \ \ c_r^{-1}\int_{0}^x 1 \ \mathrm{d}s,\end{equation}
for $x\in [0,r]$, where
\begin{equation} \label{c-def} c_r \ = \ \int_0^r\left(\frac{s}{1-s}+1\right)  \mathrm{d}s \ = \ -\log(1-r) \ <\infty. \end{equation}
 See Figure \ref{rys1}. Finally, let $m_r=\mu_r+\nu_r$. Our main result is the following.

\begin{thm}\label{main_theorem}For every $r\in (0,1)$, the measure $m_r$ is an extreme point of $\mathcal{C}$. \end{thm}

For notational simplicity, fix any $r\in(0,1)$. Although the choice of $r$ is arbitrary, we shall not highlight it any further. 
The main tools and short outline of the proof are provided in the next section. 

\section{Main tools and outline of the proof}

A few more results and definitions will be needed in order to continue the discussion. For the sake of clarity, we split the material into a few separate parts.

\subsection{A structure of coherent distributions} For a Borel measure $\pi$  on $[0,1]^2$, we will write
$\pi^x$ and $\pi^y$ for the corresponding marginal measures of $\pi$ on $[0,1]$. In other words, we have 
\begin{center}$\pi^x(A) = \mu(A\times [0,1])$ \ \ \ \ and  \ \ \ \  $\pi^y(B)=\mu([0,1]\times B)$\end{center}
 for all Borel subsets $A,B\in \mathcal{B}([0,1])$. We begin by recalling the definition of the family $\mathcal{R}$, as introduced in \cite{asymp}. This family is closely related to $\mathcal{C}$ and crucial for our analysis.

\begin{defi}\label{R-set}
Let $\mathcal{R}$ be the collection of all ordered pairs $(\mu, \nu)$ of nonnegative Borel measures on $[0,1]^2$, satisfying
$$\int_{A}(1-x) \ \mathrm{d}\mu^x  \  \ = \  \ \int_{A} x \ \mathrm{d}\nu^x,$$
and
$$\int_{B}(1-y)  \ \mathrm{d}\mu^y  \ \ =  \ \ \int_{B} y \ \mathrm{d}\nu^y,$$
for any Borel subsets $A,B \in \mathcal{B}([0,1])$. 
\end{defi}

Next, we have the following fact -- see \cite{asymp} or \cite{B1} (with a slightly  different formulation).

\begin{prop} \label{m=mu+nu} Let $m$ be a probability measure on $[0,1]^2$. Then $m$ is coherent if and only if there exists $(\mu, \nu)\in \mathcal{R}$ such that $m=\mu+\nu$. 
\end{prop}

Hereafter, for $m\in \mathcal{C}$ and any pair $(\mu, \nu) \in \mathcal{R}$ with $m=\mu+\nu$, we will shortly  say that  $(\mu, \nu)$ is the representation of a coherent distribution $m$. We have the following fact. 

\begin{prop} The pair $(\mu_r, \nu_r)$ is a representation of a coherent distribution $m_r$.
\end{prop}

\begin{proof} We need to check that $(\mu_r, \nu_r)\in \mathcal{R}$.  	First, let us note that $\mu_r^x=\mu_r^y$ since the measure $\mu_r$ is concentrated on the diagonal $\Delta_r$. Secondly, we can easily check that $\nu_r^x=\nu_r^y$ and this common measure is proportional to $\mathcal{U}(r)$, the uniform distribution on $[0,r]$:
\begin{equation} \label{nu^x=U(r)} (c_r/r)\cdot \nu_r^x \ = \ \mathcal{U}(r), \end{equation}
where $c_r$ is defined as in (\ref{c-def}). Hence, in the light of Definition \ref{R-set}, we will be done if we verify that
$$\int_{A}(1-x) \ \mathrm{d}\mu_r^x   \ =   \ \int_{A} x \ \mathrm{d}\nu_r^x,$$
for any Borel subset $A \in \mathcal{B}([0,r])$.  However, using (\ref{mur-defi}) and (\ref{nur-defi}), we get
$$\int_{A}(1-x) \ \mathrm{d}\mu_r^x  \    =   \ \frac{1}{c_r} \int_{A}(1-x)\frac{x}{1-x} \ \mathrm{d}x,$$
and
$$ \int_{A} x \ \mathrm{d}\nu_r^x \  = \ \frac{1}{c_r}\int_A x \ \mathrm{d}x,$$
which completes the proof.
\end{proof}

It is important to distinguish some special types of representations with two additional significant features.
For two Borel measures $\pi_1$, $\pi_2$ supported on the unit square $[0,1]^2$, we
write $\pi_1\le \pi_2$ if the estimate $\pi_1(A) \le \pi_2(A)$ holds for all $A \in \mathcal{B}([0,1]^2)$.

\begin{defi} \label{Un-Mi} Let $m\in \mathcal{C}$.  We say that the representation $(\mu, \nu)$ of $m$ is
\smallskip

$\cdot$ \emph{unique}, if for every $(\tilde{\mu}, \tilde{\nu})\in \mathcal{R}$ with $m=\tilde{\mu}+\tilde{\nu}$, we have $\tilde{\mu}=\mu$ and $\tilde{\nu}=\nu$;

\smallskip

$\cdot$ \emph{minimal}, if for all $(\tilde{\mu}, \tilde{\nu})\in \mathcal{R}$ with $\tilde{\mu}\le \mu$ and $\tilde{\nu}\le \nu$, there exists $\alpha \in [0,1]$ such that $(\tilde{\mu}, \tilde{\nu}) = \alpha \cdot (\mu, \nu)$. \end{defi} 

As carefully explained in \cite{asymp}, these conditions allow to state a complete description of the class $\mathrm{ext}(\mathcal{C})$.

\begin{thm} \label{ext(C)-char} Let $m$ be a coherent distribution on $[0, 1]^2$. Then $m$ is extremal if and only
if the representation of $m$ is unique and minimal. 
\end{thm}

By the above statement, our main result, Theorem \ref{main_theorem}, will be established once we have shown that representation $(\mu_r, \nu_r)$ of $m_r$ is indeed unique and minimal. These two properties will be proved in Section 3 and Section 4 below. 

\subsection{On the idea behind the special measure $m_r$} 
For $\eta \in [0,2]$, let $T_{\eta}:[0,1]\rightarrow [0,1]$ denote the celebrated 'tent map' function, defined by $T_{\eta}(x)=\eta \min\{x, 1-x\}.$
For our purposes, only the case $\eta=2$ will be needed; the corresponding function  $T_2$ is sometimes referred to as a 'full tent map'. After picking a starting point $x_0\in [0,1]$, map $T_2$ defines a discrete-time dynamical system via the recurrence relation

$$x_{n+1} \ = \ T_2(x_n) \ = \  \begin{cases}
2x_n & \mbox{if } \ x_n\le \frac{1}{2},\smallskip\\
\displaystyle 2(1-x_n) & \mbox{if } \ x_n>\frac{1}{2},
\end{cases}$$
for $n=0,1,2,\dots$ Now, if $x_0$ is irrational, then the resulting sequence  $(x_n)_{n=0}^{\infty}$  becomes injective.  To describe the dynamics, one often makes use of the so-called cobweb plot (or Verhulst diagram):  put points successively and alternately on the main diagonal and on  the graph of function $T_2$, according to the following algorithm:  
\smallskip

\begin{enumerate}[label=\arabic*.] \addtocounter{enumi}{-1}
\item place a starting point $(x_0, x_0)$ on the diagonal,
\end{enumerate}

then, for $n=1,2,\dots$, repeat
\begin{enumerate}[label=\arabic*.]
\item add a point $\left(x_{n-1}, T_2(x_{n-1})\right)$ on the graph of function $T_2$,
\item place a  point $(x_n, x_n)$ on the diagonal and increase $n$ by $1$. 
\end{enumerate}
\smallskip

Let us consider a simple extension of this procedure: apart from drawing points on the diagram, we will also assign specific weights to each new location. Namely, for a fixed huge integer $N$, place (red) mass ${x_n}/({1-x_n})$ at every point $(x_n, x_n)$ on the diagonal, $n=1,2,\dots,N-1$, and the (blue) mass $1$ for each point on the graph of function $T_2$. Finally, assign weight zero (and a purple color) to the initial and terminal points $(x_0,x_0)$, $(x_N,x_N)$. See Figure \ref{rys2}. 

\begin{figure}[H] 
\centering
\includegraphics[width=50mm]{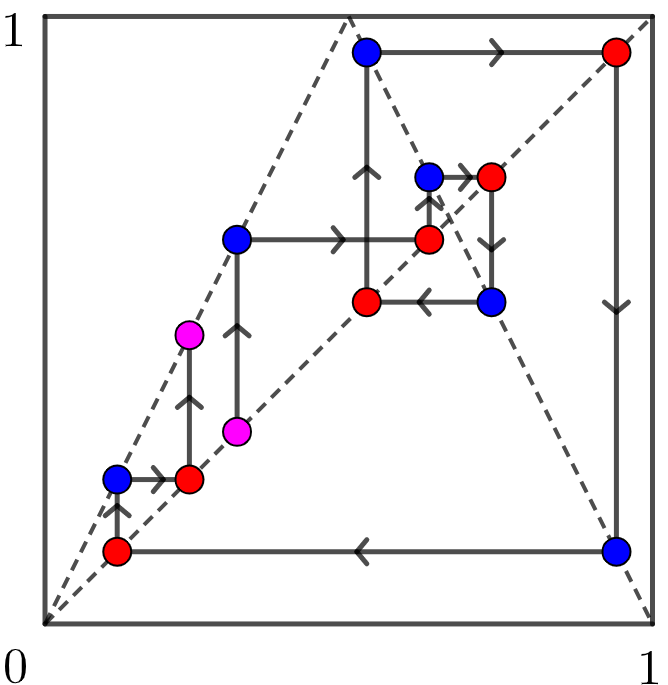}
\vspace{0.2cm}
\caption{An example of a cobweb diagram showing the first iterations of a full tent map system.  Purple points are endpoints of an  'axial path' obtained in this process. Red points represent masses placed on the diagonal,
while blue points indicate masses added on the graph of function $T_2$.}
\label{rys2} 
\end{figure}

We are ready to discuss (informally) the intuition which has led us to the measure $m_r$. Take the initial subsequence $(x_0, x_0), (x_0, x_1), (x_1, x_1), \dots, (x_{N-1}, x_{N})$, for some big number $N$. Let $\widetilde{m}^N$ be a probability distribution supported on those points, determined by their normalized weights. Then, choose a random point $P$ on the square (distributed according to $\widetilde{m}^N$) and let $(X,Y)$ be its coordinates. Let $E$ denote the event that $P$ belongs to the diagonal. Then for any $1\le i,j \le N-1$ we may write
\begin{equation} \label{Px-coh} \mathbb{P}(E|X=x_i) \ = \ \frac{\frac{x_i}{1-x_i}}{\frac{x_i}{1-x_i}+1} \ =  \ x_i  \end{equation}
and
\begin{equation} \label{Py-coh} \mathbb{P}(E|Y=x_j) \ = \  \frac{\frac{x_j}{1-x_j}}{\frac{x_j}{1-x_j}+1} \ = \ x_j, \end{equation}
where the middle fraction is the ratio of the red mass to the total mass in the line. Intuitively, equations (\ref{Px-coh}) and  (\ref{Py-coh}) show that 
$\widetilde{m}^N$ is 'almost coherent': the representation of this measure is the pair of its red and blue component. 
Now it is natural to try to carry out a limiting argument: as $N$ goes to infinity, red points uniformly cover the entire diagonal and blue points fill up the graph of $T_2$. It seems plausible that $\widetilde{m}^N$ converges to a certain limiting measure $\widetilde{m}\in \mathrm{ext}(\mathcal{C})$ with no atoms. Unfortunately, this is not the case: as $\frac{x}{1-x}\rightarrow \infty$ for $x\rightarrow 1$, the red mass explodes uncontrollably near the corner $(1,1)$. 

Nevertheless, there is an easy way to overcome this obstacle. Rather than working with the transform $T_2$, we will simply employ its rescaled version $t_r:[0,r]\rightarrow [0,r]$, $$t_r(x)=2\min\{x, r-x\}.$$
Observe that images of transform $T_2$ and its shrunken version $t_r$ are homothetic. Thus, if we just repeat the former reasoning with $t_r$ acting on $[0,r]$, instead of $T_2$ acting on $[0,1]$, then any point mass added to the diagram will not exceed $\frac{r}{1-r}<\infty$. This leads us to the consideration of
a new sequence $\widetilde{m}^N_r$,  defined analogously to  $\widetilde{m}^N$, of 'almost coherent' and 'resembling extreme points' measures on $[0,r]^2$.
Letting $N\to \infty$, looking at the red and blue components of $\widetilde{m}^N_r$ separately, we see that red mass approaches to $\mu_r$ and blue mass tends to $\nu_r$. This motivates our claim:  $m_r = \mu_r + \nu_r$ is an extreme coherent distribution.

\subsection{Some basic facts from the theory of dynamical systems} To carry out a formally strict justification, we will need the following classical definition of a measure-preserving and ergodic transformation. By $B\triangle C$ we will denote the symmetric difference $(B\setminus C)\cup (C \setminus B)$ of sets $B$ and $C$.

\begin{defi} Fix a constant $s>0$ and let $\pi$ be a Borel probability measure on $[0,s]$. Let $T:[0,s]\rightarrow [0,s]$ be a Borel measurable function.
We say that the transformation $T$ is
\smallskip

$\cdot$ \emph{measure-preserving} (or equivalently, that $\pi$ is \emph{$T$-invariant}), if 
$$\pi(T^{-1}(A)) \ = \ \pi(A),$$  for all Borel subsets $A\in \mathcal{B}([0,s])$;

\smallskip

$\cdot$ \emph{ergodic}  (or equivalently, that $\pi$ is an \emph{ergodic measure}), if $T$ is measure-preserving and
any set $A\in \mathcal{B}([0,s])$ with $\pi \left(T^{-1}(A) \triangle A\right)=0$ must satisfy $\pi(A)\in \{0,1\}$.
\end{defi}

Therefore, if a measure-preserving transformation $T$ is not ergodic, then we can find some $A\in \mathcal{B}([0,s])$  with $T^{-1}(A)=A$ and $\pi(A)\in (0,1)$. 
In principle, this means that restrictions $T|_A:A\rightarrow A$ and $T|_{A^c}:A^c\rightarrow A^c$ define a non-trivial decomposition of the map $T$ into two simpler subsystems. Noteworthy, there is a subtle similarity between ergodicity of transformations and minimality of representations from Definition \ref{Un-Mi}. Eventually, we need to make more precise the statement: 'full tent map $T_2$ covers $[0,1]$ uniformly'.

\begin{prop}\label{U(1)-ERG}
$\mathcal{U}(1)$, the uniform distribution on $[0,1]$, is $T_2$-invariant and ergodic.
\end{prop}

Proposition \ref{U(1)-ERG} is a well-known fact in the theory of dynamical systems; hence, we omit its proof, referring the interested reader to  \cite{Buzzi2009}, for example. In what follows, we will need a simple reformulation of this statement for $r<1$.

\begin{cor}\label{U(r)-tr-inv} $\mathcal{U}(r)$, the uniform distribution on $[0,r]$, is $t_r$-invariant and ergodic.
\end{cor}

\begin{proof} For $x\in [0,r]$, we have
\begin{equation} \label{tr-1} t_r^{-1}(x) \ = \ r\cdot T_2^{-1}(x/r). \end{equation}
To check that $\mathcal{U}(r)$ is $t_r$-invariant, we need to show that

\begin{equation} \label{tr-inv}  \lvert t_r^{-1}(A) \rvert \cdot \frac{1}{r} \ = \  \lvert A \rvert  \cdot \frac{1}{r},  \end{equation} 
for all $A\in \mathcal{B}([0,r])$, where  $ \lvert \cdot \rvert$ stands for the Lebesgue measure on the line. Yet, by (\ref{tr-1}) and Proposition \ref{U(1)-ERG} (as $\mathcal{U}(1)$ is $T_2$-invariant), we can write
$$\lvert t_r^{-1}(A) \rvert \cdot r^{-1} \ = \ \lvert T_2^{-1}(A/r)\rvert \ = \ \lvert A/r \rvert \ = \ \lvert A \rvert \cdot r^{-1},$$
which proves (\ref{tr-inv}). Next, if $t_r$ was not ergodic, then we could find a set $B\in \mathcal{B}([0,r])$  with $t_r^{-1}(B)=B$ and 
$\lvert B \rvert \in (0,r)$. Consequently, by (\ref{tr-1}) we would obtain
$$T_2^{-1}(B/r) \ = \ \frac{1}{r}\cdot t_r^{-1}(B) \ = \ B/r,$$
while $B/r\in \mathcal{B}([0,1])$ and $\lvert B/r \rvert \in (0,1)$. This would indicate that the map $T_2$ is not ergodic, contradicting Proposition \ref{U(1)-ERG}.
\end{proof}


\section{Proof of uniqueness}

Now we will show that $(\mu_r, \nu_r)$ is the only representation of measure $m_r$. To this end, assume that 
$m_r  =   \widetilde{\mu}_r+\widetilde{\nu}_r$ is some representation $(\widetilde{\mu}_r, \widetilde{\nu}_r)\in \mathcal{R}$; we will demonstrate that $\widetilde{\mu}_r=\mu_r$ and $\widetilde{\nu}_r=\nu_r$. 
From now on, to simplify the notation, we will skip the index $r$ and write  $\mu$, $\nu$, $\widetilde{\mu}$, $\widetilde{\nu}$, $m$ 
instead of  $\mu_r$, $\nu_r$, $\widetilde{\mu}_r$, $\widetilde{\nu}_r$ and $m_r,$
respectively. 
First, fix an arbitrary $A \in \mathcal{B}([0, r])$. By Definition \ref{R-set}, we have 

$$ \int_{A} 1 \ \mathrm{d}\mu^x  \ = \ \int_{A} x \ (\mathrm{d}\nu^x+\mathrm{d}\mu^x) \ = \  
 \int_{A} x \ \mathrm{d}m^x, $$
 and
$$\int_{A} 1 \ \mathrm{d}\widetilde{\mu}^x  \ = \ \int_{A} x \ (\mathrm{d}\widetilde{\nu}^x+\mathrm{d}\widetilde{\mu}^x) \ = \  
 \int_{A} x \ \mathrm{d}m^x,$$
 so $\mu^x(A)=\widetilde{\mu}^x(A)$. Similarly, we can deduce that $\mu^y=\widetilde{\mu}^y$, and hence the marginal distributions of $\mu$ and $\widetilde{\mu}$ are  equal. Because  $\mu +\nu = \widetilde{\mu} + \widetilde{\nu}$, 
we also get 
\begin{equation} \label{nu-x,y-marg} \nu^x=\widetilde{\nu}^x \ \ \ \ \ \ \mathrm{and} \ \ \ \ \ \ \nu^y=\widetilde{\nu}^y.\end{equation}
However, directly from (\ref{nur-defi}), we infer that $\nu^x= \nu^y$ and both these measures are proportional to the uniform distribution $\mathcal{U}(r)$. Let us inspect the finite Borel signed measure $\delta = \widetilde{\nu} - \nu$. 
By the Jordan decomposition theorem, we can find two nonnegative measures $\delta_+$ and $\delta_-$, which are mutually singular and such that 
\begin{equation} \label{d+d-} \delta \ = \ \delta_+ \ - \ \delta_-. \end{equation}
As a consequence of (\ref{nu-x,y-marg}), the marginal measures $\delta^x$ and $\delta^y$ are identically equal to zero.
Therefore, using (\ref{d+d-}) we obtain
\begin{equation} \label{delta+x-x}  \delta^x_+=\delta^x_- \ \ \ \ \ \ \mathrm{and} \ \ \ \ \ \ \ \delta^y_+=\delta^y_-. \end{equation}
The following observation is elementary, nonetheless very helpful.

\begin{prop} \label{d<m,d<n} We have $\delta_-\le \nu $ and $\delta_+\le \mu.$ \end{prop}

\begin{proof} Since $\delta_+$ and $\delta_-$ are mutually singular, there exist two disjoint Borel sets $S_+, S_-$ with $S_+\cup S_-=[0,r]^2$, such that
\begin{equation} \label{d+S+} \delta_+(A) \ = \ \delta_+(A\cap S_+), \end{equation}
\begin{equation} \label{d-S-} \delta_-(A) \ = \ \delta_-(A\cap S_-), \end{equation}
for all Borel subsets $A\in \mathcal{B}([0,r]^2)$. To show that $\delta_-\le \nu$, it is sufficient to check that 
$$\delta_-(B\cap S_-) \ \le \ \nu(B\cap S_-),$$
for every choice of $B\in \mathcal{B}([0,r]^2)$. By (\ref{d+d-}) and  (\ref{d+S+}), we can write
$$-\delta_-(B\cap S_-) \ = \ \delta(B\cap S_-) \ = \ \widetilde{\nu}(B\cap S_-) - \nu(B \cap S_-),$$
which yields
$$ \nu(B \cap S_-) \ = \ \delta_-(B\cap S_-) + \widetilde{\nu}(B\cap S_-) \ \ge \ \delta_-(B\cap S_-). $$
In the same way,  by (\ref{d+d-}) and  (\ref{d-S-}), we get
\begin{align*}
\delta_+(B\cap S_+) + (\widetilde{\mu}-\mu)(B\cap S_+) & =  \delta(B\cap S_+) + (\widetilde{\mu}-\mu)(B\cap S_+)\\
 & =   (\widetilde{\nu}+\widetilde{\mu})(B\cap S_+) - (\nu+\mu)(B\cap S_+) \ = \ 0, 
\end{align*}
which results in
$$\mu(B\cap S_+) \ = \  \delta_+(B\cap S_+) + \widetilde{\mu}(B\cap S_+) \ \ge \ \delta_+(B\cap S_+),$$ 
and proves that $\delta_+\le \mu$.
\end{proof}

A central component of our argument is the next conclusion.

\begin{figure}[H] 
\centering
\includegraphics[width=44mm]{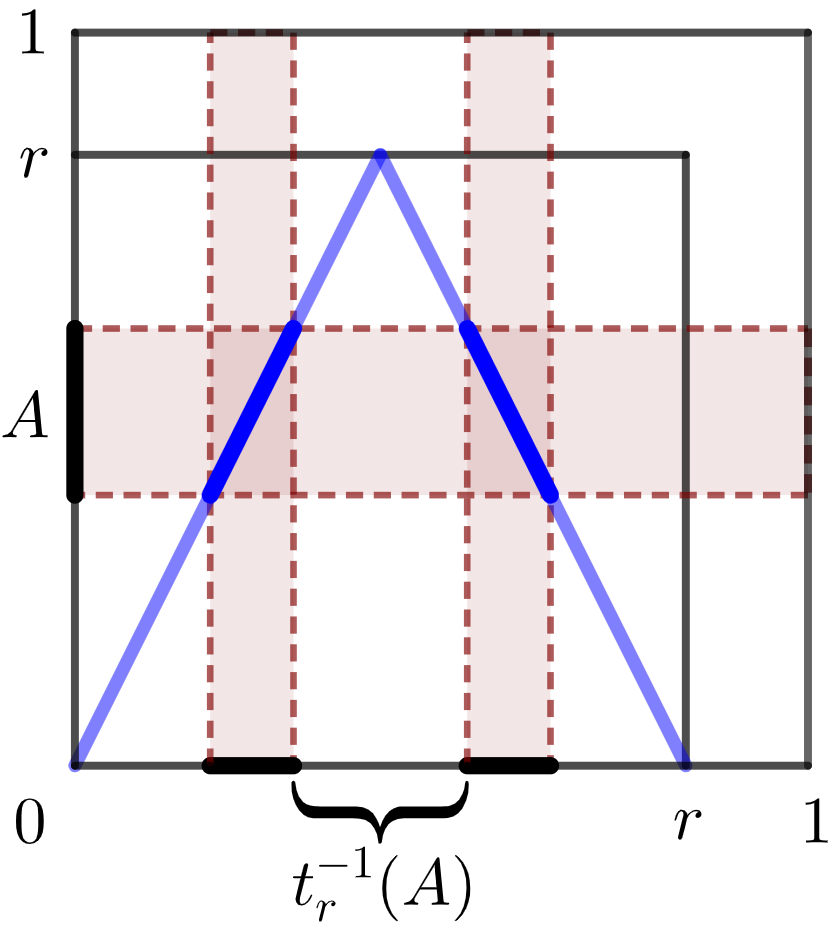}
\vspace{0.2cm}
\caption{The preimage $t_r^{-1}(A)$ of a Borel set $A$ is made up of two non-overlapping subsets with twice as small Lebesgue measure. }
\label{rys3} 
\end{figure}

\begin{prop} \label{a-d-tr-inv} Assume that $\delta_-$ is not  identically equal to zero and let
 $$\alpha_{-}   \ = \ \frac{1}{\delta_-([0,r]^2)}$$ 
denote its norming constant. Then \ $\alpha_{-}\cdot \delta_-^x$ \ is a \  $t_r$-invariant  measure. \end{prop}

\begin{proof} Since $\delta_-$ is not a zero measure, we get
$$\delta_-([0,r]^2) \ = \  \delta_-^x([0,r])$$
and hence $\alpha_{-}\cdot \delta_-^x$ is a probability measure. Now, by Proposition \ref{d<m,d<n} we have
$$\delta_-\le \nu \ \ \ \ \ \ \mathrm{and} \ \ \ \ \ \ \delta_+\le \mu,$$
which leads to
\begin{equation} \label{supp(d-)} \mathrm{supp}(\delta_-) \ \subset \ \mathrm{supp}(\nu) \ = \ \Gamma_r  \end{equation}
and
\begin{equation} \label{supp(d+)}  \mathrm{supp}(\delta_+) \ \subset \ \mathrm{supp}(\mu) \ = \ \Delta_r, \end{equation}
respectively. Directly from (\ref{supp(d+)}), we see that $ \mathrm{supp}(\delta_+)$ is a subset of the main diagonal, and therefore
$\delta_+^x=\delta_+^y$. Consequently, due to (\ref{delta+x-x}),  we get the identities
\begin{equation} \label{d-x=d-y} \delta_-^x \ = \ \delta_+^x \ = \ \delta_+^y \ = \ \delta_-^y. \end{equation}
Finally, take an arbitrary Borel set $A\in \mathcal{B}([0,r])$. From (\ref{supp(d-)}) we know that $\mathrm{supp}(\delta_-)$ is contained 
in $\Gamma_r$, i.e. in the graph of function $t_r$, and so
\begin{equation}
\begin{split}
 \label{d-x(tr-1A)} \delta_-^y(A) & =  \delta_-([0,r] \times A) 
 = \delta_-\Big[ \big( [0,r]\times A \big) \cap \Gamma_r \Big]\\
 & = \ \delta_-^x \big( \{x\in [0,r]  : \ t_r(x)\in A \} \big) \ = \ \delta_-^x\left( t_r^{-1}(A) \right),
 \end{split}
\end{equation}
see Figure \ref{rys3}. Combining (\ref{d-x=d-y}) and (\ref{d-x(tr-1A)}),  we conclude that
$$\alpha_{-}\cdot \delta_-^x(A) \ \ =  \ \ \alpha_{-}\cdot \delta_-^y(A) \ \ =  \ \  \alpha_{-} \cdot \delta_-^x\left( t_r^{-1}(A) \right),$$
which is the claim.
\end{proof}

We will require the following folklore remark. For the sake of completeness, we also provide its straightforward proof.

\begin{rem} \label{ERGO-rem} Take $s>0$ and let $T:[0,s]\rightarrow [0,s]$ be a Borel measurable transformation.
Let $\pi$ and $\sigma$ be two $T$-invariant probability measures on $\mathcal{B}([0,s])$ and suppose that $\pi$ is additionally ergodic. 
If $\sigma \ll \pi$ (i.e., if $\sigma$ is absolutely continuous with respect to $\pi$), then $\sigma=\pi$.
 \end{rem}

\begin{proof} Let $\rho$ denote the Radon--Nikodym derivative of $\sigma$ with respect to $\pi$: for any $C\in \mathcal{B}([0,s])$
we have $\sigma(C) = \int_C \rho \ \mathrm{d}\pi$. Let us distinguish the set $A=\{\rho<1\}$. Because $\pi$ and $\sigma$ are $T$-invariant, we may write
\begin{equation} \label{pi-T-inv} \pi \left(A \setminus T^{-1}\left( A \right) \right) \ = \ \pi \left(T^{-1}\left( A\right) \setminus A \right),  \end{equation}
and
\begin{equation} \label{sig-T-inv} \sigma \left(A \setminus T^{-1}\left( A\right) \right) \ = \ \sigma \left(T^{-1}\left( A\right) \setminus A \right). \end{equation}
Hence, if $\pi (A \setminus T^{-1}( A) )>0$, then the very definition of $A$ we get
\begin{equation} \label{PI>int} \pi (A \setminus T^{-1}( A) ) \ > \ \int_{A \setminus T^{-1}(A)} \rho \ \mathrm{d}\pi \ = \  \sigma (A \setminus T^{-1}( A)). \end{equation}
Analogously, we show that
\begin{equation} \label{SIG=int} \sigma \left(T^{-1}\left( A\right) \setminus A \right) \ = \ \int_{T^{-1}\left( A\right) \setminus A} \rho \ \mathrm{d}\pi \ \ge \   \pi \left(T^{-1}\left( A\right) \setminus A \right). \end{equation}
Combining (\ref{sig-T-inv}), (\ref{PI>int}) and (\ref{SIG=int}), we obtain
$$ \pi \left(A \setminus T^{-1}\left( A \right) \right) \ > \ \pi \left(T^{-1}\left( A\right) \setminus A \right),$$
which contradicts (\ref{pi-T-inv}) and proves that $\pi (A \triangle T^{-1}( A))=0$.
Since measure $\pi$ is ergodic, we conclude that $\pi(A)\in \{0,1\}$. But $\pi(\{\rho<1\})=1$ is impossible as $\pi$ and $\sigma$ are both probabilistic.
Therefore we must have $\pi(\{\rho \ge1\})=1$. Again, since $\pi$ and $\sigma$ are probabilistic, we deduce that $\rho=1$ $\pi$-almost everywhere and thus $\sigma=\pi$.
\end{proof}

We proceed to the the primary objective of this section.

\begin{proof}[Proof of uniqueness] Let us assume that $\delta_-$ is not identically equal to zero. Then 
by Proposition \ref{d<m,d<n} and equality (\ref{nu^x=U(r)}), we know that
$$\delta_-^x\le \nu^x \ \ \ \ \ \ \mathrm{and} \ \ \ \ \ \ \nu^x = \frac{r}{c_r}\cdot \mathcal{U}(r),$$
and hence $\alpha_-\cdot \delta_-^x \ \ll \mathcal{U}(r)$. Moreover,  by Proposition \ref{a-d-tr-inv} and Corollary \ref{U(r)-tr-inv}, 
we note that  $\alpha_-\cdot \delta_-^x$ and $\mathcal{U}(r)$ are both $t_r$-invariant. But  $\mathcal{U}(r)$ is also ergodic, so Remark \ref{ERGO-rem} gives $\alpha_-\cdot \delta_-^x=\mathcal{U}(r)$. Combining this with (\ref{delta+x-x}) and Proposition \ref{d<m,d<n}, we get
\begin{equation} \label{U(r)<=mu-x} \alpha_-^{-1}\cdot \mathcal{U}(r) \ = \ \delta_-^x \ = \delta_+^x  \ \le \ \mu^x. \end{equation}
However, directly from (\ref{mur-defi}), for $s\in (0,s_0)$ we have
$$\mu^x\left([0,s] \right) \ = \  c_r^{-1}\int_{0}^s\frac{z}{1-z} \ \mathrm{d}z \ < \  \alpha_-^{-1}\cdot s, $$
provided $s_0$ is sufficiently small. This contradicts (\ref{U(r)<=mu-x}) and shows that $\delta_-$ is a zero measure. Consequently, we get $\delta=\delta_+$: $\delta$ is a nonnegative measure. But  (\ref{nu-x,y-marg}) gives $\delta^x=0$, so we must have $\delta=0$ and $\widetilde{\nu} = \nu$.  Finally, we write down the identity
$$\widetilde{\mu} \ = \ m - \widetilde{\nu} \ = \ m - \nu \ = \ \mu,$$
which gives $(\widetilde{\mu}, \widetilde{\nu})  =  (\mu, \nu)$. This proves that $(\mu, \nu)$ is a unique representation of $m$.\end{proof}

\section{Proof of minimality}

It remains to verify the minimality of the representation $(\mu, \nu)$. Suppose that
$$\widetilde{\mu}\le \mu \ \ \ \ \ \ \mathrm{and} \ \ \ \ \ \ \widetilde{\nu} \le \nu,$$
for some $(\widetilde{\mu}, \widetilde{\nu})\in \mathcal{R}$. We need to show that $(\widetilde{\mu}, \widetilde{\nu})$ is proportional to $(\mu, \nu)$.

\begin{proof}[Proof of minimality] Thanks to $\widetilde{\mu}\le \mu$, we have
\begin{equation}\label{support-tilde-mu} \mathrm{supp}(\widetilde{\mu}) \ \subset \ \mathrm{supp}(\mu) \ = \ \Delta_r.\end{equation}
In particular, $\mathrm{supp}(\widetilde{\mu})$ is a subset of the main diagonal, which implies that
\begin{equation} \label{tilde-mu-x=y} \widetilde{\mu}^x \ = \ \widetilde{\mu}^y. \end{equation}
Next, using (\ref{tilde-mu-x=y}) and Definition \ref{R-set}, we get
 \begin{equation} \label{R-set-tilde-mu} \begin{aligned} 
 \int_{A} x \ \mathrm{d}\widetilde{\nu}^x    \ =&   \   \int_{A}(1-x) \ \mathrm{d}\widetilde{\mu}^x    \\
=& \  \int_{A}(1-x) \ \mathrm{d}\widetilde{\mu}^y   \  =   \  \int_{A} x \ \mathrm{d}\widetilde{\nu}^y,
 \end{aligned} \end{equation}
for any Borel subset $A \in \mathcal{B}([0,r])$. Due to $\widetilde{\nu}\le \nu$, we have 
$$\widetilde{\nu}^x\ll \nu^x \ \ \ \ \ \ \mathrm{and} \ \ \ \ \ \ \widetilde{\nu}^y\ll \nu^y.$$
Furthermore, based on (\ref{nu^x=U(r)}), recall that $\nu^x=\nu^y$ and $\nu^x$ is proportional to $\mathcal{U}(r)$.
Applying the above observations, let $\varphi_x, \varphi_y$ denote the Radon--Nikodym derivatives of $\widetilde{\nu}^x$ and $\widetilde{\nu}^y$ with respect to $\nu^x$. Put
$$A_+ =  \left\{ \varphi_x>\varphi_y  \right\} \ \ \ \ \ \ \mathrm{and} \ \ \ \ \ \ A_- =  \left\{ \varphi_x<\varphi_y  \right\}.$$
Inserting sets $A_+, A_-$ into (\ref{R-set-tilde-mu}), we obtain
\begin{equation} \label{int-A+/-} \int_{A_{\pm}} x \left(\varphi_x - \varphi_y \right)  \mathrm{d}\nu^x  \ = \ \int_{A_{\pm}}x \ \mathrm{d}\widetilde{\nu}^x  \ - \ \int_{A_{\pm}}x  \ \mathrm{d}\widetilde{\nu}^y \ = \ 0,\end{equation}
which shows that $A_+\cup A_-$ is a set of $\nu^x$-measure zero and hence $\widetilde{\nu}^x=\widetilde{\nu}^y$.
On the other hand, we have
\begin{equation} \label{support-tilde-nu} \mathrm{supp}(\widetilde{\nu}) \ \subset \ \mathrm{supp}(\nu) \ = \ \Gamma_r, \end{equation}
so $\mathrm{supp}(\widetilde{\nu})$ is a subset of the graph of the function $t_r$.
Thus, we can simply repeat the reasoning from the proof  of Proposition \ref{a-d-tr-inv} -- rewriting the equation (\ref{d-x(tr-1A)}) with $\delta_-$ replaced by $\widetilde{\nu}$, we get that $\widetilde{\nu}$ is a  $t_r$-invariant  measure (up to proportionality). Then, again by Remark \ref{ERGO-rem}, 
\begin{equation} \label{tilde-nu-x-alpha*U} \widetilde{\nu}^x \ =\  \alpha \cdot \mathcal{U}(r) \ = \ \alpha' \cdot \nu^x,\end{equation}
 for some factors $\alpha, \alpha' \ge 0$. Combining (\ref{tilde-nu-x-alpha*U}) with Definition \ref{R-set}, we get
  \begin{equation} \label{R-set-tilde-mu-2} \begin{aligned} 
 \int_{A}(1-x) \ \mathrm{d}\widetilde{\mu}^x   \ =&   \   \int_{A} x \ \mathrm{d}\widetilde{\nu}^x     \\
=& \ \alpha' \cdot \int_{A} x \ \mathrm{d}\nu^x   \  =   \  \alpha'\cdot \int_{A}(1-x) \ \mathrm{d}\mu^x,
 \end{aligned} \end{equation}
for any Borel subset $A \in \mathcal{B}([0,r])$. Just as in (\ref{int-A+/-}), plugging the Radon--Nikodym derivative of $\widetilde{\mu}^x$
with respect to $\mu^x$ into (\ref{R-set-tilde-mu-2}), we check that $\widetilde{\mu}^x=\alpha'\cdot \mu^x$. Altogether with  (\ref{tilde-nu-x-alpha*U}), (\ref{support-tilde-mu}) and (\ref{support-tilde-nu}), this yields $(\widetilde{\mu}, \widetilde{\nu})=\alpha'\cdot(\mu, \nu)$. The proof of the minimality is hence complete.
\end{proof} 


\setlength{\baselineskip}{2ex}

\bibliographystyle{plain}
\bibliography{ERGObib}

\end{document}